\documentclass[10pt,reqno]{amsart}

\usepackage[utf8x]{inputenc}
\usepackage[english]{babel}
\usepackage{amsmath,amsfonts,amssymb,amsthm,shuffle}
\usepackage[T1]{fontenc}
\usepackage{lmodern}
\usepackage{mathtools}

\usepackage[top=3.5cm,bottom=3.5cm,left=3.6cm,right=3.6cm]{geometry}

\usepackage[dvipsnames]{xcolor}
\usepackage[hyperindex=true,frenchlinks=true,colorlinks=true,
citecolor=Mahogany,linkcolor=DarkOrchid,urlcolor=Tan,linktocpage,
pagebackref=true]{hyperref}

\usepackage{tikz}
\usetikzlibrary{shapes}
\usetikzlibrary{fit}
\usetikzlibrary{decorations.pathmorphing}

\usepackage{dsfont}
\usepackage{wasysym}
\usepackage{stmaryrd}
\usepackage{cite}
\usepackage{subfigure}
\usepackage{multirow}
\usepackage{enumitem}
\usepackage{multicol}

\title{H\"olderian weak invariance principle under Maxwell and Woodroofe condition}
\keywords{stationary sequences; Maxwell and Woodroofe condition; 
H\"older spaces; invariance principle.}
\date{\today}
\address{Normandie Universit\'e, Universit\'e de Rouen, 
Laboratoire de Math\'ematiques Rapha\"el Salem,
CNRS, UMR 6085, Avenue de l'universit\'e, BP 12, 
76801 Saint-Etienne du Rouvray Cedex, 
France.}
\email{davide.giraudo1@univ-rouen.fr}

\numberwithin{equation}{section}
\newtheorem{Theorem}{Theorem}[section]
\newtheorem{Definition}[Theorem]{Definition}
\newtheorem{Proposition}[Theorem]{Proposition}
\newtheorem{Lemma}[Theorem]{Lemma}
\newtheorem{Corollary}[Theorem]{Corollary}

\theoremstyle{remark}
\newtheorem{Remark}[Theorem]{Remark}

\tikzstyle{Vertex}=[circle,draw=LimeGreen!80,fill=LimeGreen!8,
inner sep=1pt,minimum size=2mm,line width=1pt,font=\scriptsize]
\tikzstyle{Node}=[Vertex,draw=RoyalBlue!80,fill=RoyalBlue!8,inner sep=1.5pt]
\tikzstyle{Leaf}=[rectangle,draw=Black!70,fill=Black!16,
inner sep=0pt,minimum size=1mm,line width=1.25pt]
\tikzstyle{Edge}=[Maroon!80,cap=round,line width=1pt]
\tikzstyle{Mark1}=[draw=BrickRed!80,fill=BrickRed!8]
\tikzstyle{Mark2}=[draw=BurntOrange!80,fill=BurntOrange!8]
\tikzstyle{EdgeRew}=[->,RedOrange!80,cap=round,thick]

\newcommand{\norm}[1]{\left\lVert #1 \right\rVert}

\newcommand \ens[1]{\left\{ #1\right\}}
\newcommand \R{\mathbb R}

\newcommand \Z{\mathbb Z}

\newcommand \abs[1]{\left|#1\right|}

\newcommand{\Hide}[1]{}

\begin{document}
\title{H\"olderian weak invariance principle under the Maxwell and Woodroofe condition}

\author{Davide Giraudo }

\date{\today}

\keywords{Invariance principle, martingales, Hölder spaces, strictly stationary process}

\subjclass[2010]{60F05; 60F17}

 \begin{abstract}
  We investigate the weak invariance principle in H\"older spaces under 
  some reinforcement of the Maxwell and Woodroofe condition. Optimality 
  of the obtained condition is established. 
 \end{abstract}
 
 \maketitle
 
 \section{Introduction and main results}
 
 Let $(\Omega, \mathcal F, \mu)$ be a probability space and let $T \colon  \Omega \to \Omega$
  be a measure-preserving bijective and bi-measurable map. Let $\mathcal M$ be a 
  sub-$\sigma$-algebra of $\mathcal F$ such that $T\mathcal M\subset\mathcal M$.
  If $f\colon\Omega\to \R$ a measurable function, we denote 
  $S_n(T,f):=\sum_{j=0}^{n-1}f\circ T^j$ and 
  \begin{equation}\label{def_partial_sum_process}
   W(n,f,T,t):=S_{[nt]}(T,f)+(nt-[nt])f\circ T^{[nt]}.   
  \end{equation}
  We shall write $S_n(f)$ and $W(n,f,t)$ for simplicity, except when $T$ is replaced 
  by $T^2$. 
  
  An important problem in probability theory is the understanding of 
  the asymptotic behavior of the process $(n^{-1/2} W(n,f,t),t\in [0,1])_{n\geqslant 1}$. 
  Conditions on the quantities $\mathbb E[S_n(f)\mid T\mathcal M]$ 
  and $S_n(f)-\mathbb E[S_n(f)\mid T^{-n}\mathcal M]$ have been investigated.
  The first result in this direction was obtained by Maxwell and Woodroofe 
  \cite{MR1782272}: if $f$ is $\mathcal M$-measurable and 
  \begin{equation}\label{MW_adapted}
   \sum_{n=1}^{+\infty}\frac{\norm{\mathbb E\left[S_n(f)\mid \mathcal M\right]}_2}{n^{3/2} }
   <+\infty,     
  \end{equation}
  then $(n^{-1/2}S_n(f) )_{n\geqslant 1}$ converges in distribution to $\eta^2 N$, 
  where $N$ is normally distributed and independent of $\eta$. 
  Then Voln\'y \cite{MR2283254} proposed a method to treat the nonadapted case. 
  Peligrad and Utev \cite{MR2123210} proved the weak invariance principle 
  under condition \eqref{MW_adapted}. The nonadapted case was addressed 
  in \cite{MR2344817}. Peligrad and Utev also showed that condition \eqref{MW_adapted}
  is optimal among conditions on the growth of the sequence 
  $\left(\norm{ \mathbb E\left[S_n(f)\mid \mathcal M \right]}_2\right)_{n\geqslant 1}$:
  if 
  \begin{equation}\label{weaked_MW}
   \sum_{n=1}^{+\infty}a_n\frac{\norm{\mathbb E\left[S_n(f)\mid \mathcal M\right]}_2}{n^{3/2} }<\infty
  \end{equation}
 for some sequence $(a_n)_{n \geqslant 1}$ converging to $0$, 
 the sequence $(n^{-1/2}S_n(f))_{n \geqslant 1}$ 
 is not necessarily stochastically bounded (Theorem~1.2. of \cite{MR2123210}). 
 Voln\'y constructed  \cite{MR2679961} an example satisfying \eqref{weaked_MW} and 
 such that the sequence $\left(\norm{S_n(f)}_2^{-1}S_n(f)\right)_{n\geqslant 1}$ 
 admits two subsequences which converge weakly to two different distributions. 
 
 Let us denote by $\mathcal H_{\alpha}$ the space of H\"older continuous functions, that is, 
 the functions $x \colon [0,1]\to \R$ such that $\norm{x}_{\mathcal H_\alpha}:=
 \sup_{0\leqslant s<t\leqslant 1}\abs{x(t)-x(s)}/(t-s)^\alpha+\abs{x(0)}$ is finite. 
 Since the paths of Brownian motion belong almost surely to $\mathcal H_\alpha$ for 
 each $\alpha\in (0,1/2)$ as well as $W(n,f,\cdot)$, we can investigate the weak 
 convergence of the sequence $(n^{-1/2}W(n,f,\cdot))_{n\geqslant 1}$ in the 
 the space $\mathcal H_{\alpha}$, for $0<\alpha<1/2$.
 The case of i.i.d. sequences and stationary martingale difference 
 sequences have been addressed respectively by Ra\v{c}kauskas and Suquet (Theorem~1 of 
 \cite{MR2000642}) and Giraudo (Theorem~2.2 of \cite{MR3426520}).
 In this note, we focus on conditions on the sequences 
 $\left(\mathbb E[S_n(f)\mid \mathcal M]\right)_{n\geqslant 1}$ and 
 $\left(S_n(f)-\mathbb E[S_n(f)\mid T^{-n}\mathcal M]\right)_{n\geqslant 1}$.

  \begin{Theorem}\label{thm:main_result} 
   Let $p>2$ and $f\in \mathbb L^p$. If 
   \begin{equation}\label{MWp}  
    \sum_{k=1}^{+\infty} \frac{\norm{\mathbb E[S_k(f)\mid \mathcal M}_p }{k^{3/2}}<+\infty, 
    \quad \sum_{k=1}^{+\infty} \frac{\norm{S_k(f) -
    \mathbb E[S_k(f)\mid T^{-k} \mathcal M}_p }{k^{3/2}}<+\infty,
   \end{equation}
  then the sequence $\left(n^{-1/2}W(n,f)\right)_{n\geqslant 1}$ 
  converges weakly to the process $\sqrt\eta W$ in $\mathcal H_{1/2-1/p}$, where $W$ 
  is the Brownian motion and the random variable $\eta$ is independent of $W$ and 
  is given by $\eta=\lim_{n\to +\infty}\mathbb E\left[S_n(f)^2\mid\mathcal I
  \right]/n$ (where $\mathcal I$ is the $\sigma$-algebra of invariant sets and the 
  limit is in the $\mathbb L^1$ sense).
  \end{Theorem}

  Of course, if $f$ is $\mathcal M$-measurable, all the 
 terms of the second series vanish and we only have to check the convergence 
 of the first series.
 
 \begin{Remark}
  If the sequence $(f\circ T^j)_{j\geqslant 0}$ 
  is a martingale difference sequence with respect to the filtration 
  $(T^{-i}\mathcal M)$, then condition \eqref{MWp} is satisfied if and only if 
  the function $f$ belongs to $\mathbb L^p$, hence we recover the result of 
  \cite{MR3426520}. However, if the sequence $(f\circ T^j)_{j\geqslant 0}$ 
  is independent, \eqref{MWp} is stronger than the sufficient condition 
  $t^p\mu\ens{\abs f>t}\to 0$. This can be explained by the fact that 
  the key maximal inequality \eqref{maximal_inequality} does not include 
  the quadratic variance term which appears in the 
  martingale inequality. In Remark~1 (after the proof of 
  Theorem~1) in \cite{MR2255301}, a version of \eqref{maximal_inequality} 
  with this term is obtained. In our context it seems that it 
  does not follow from an adaptation of the proof. 
 \end{Remark}

 \begin{Remark}\label{rem:comparison}
  In \cite{MR3426520}, the conclusion of Theorem~\ref{thm:main_result} was 
  obtained for an $\mathcal M$-measurable $f$ under the condition 
  \begin{equation}\label{Hannan} 
   \sum_{i=1}^\infty\norm{\mathbb E\left[f\mid T^i\mathcal M\right]
   -\mathbb E\left[f\mid T^{i+1} \mathcal M\right] }_p<\infty,
  \end{equation}
  which holds as soon as 
  \begin{equation}\label{eq:consequence_Hannan}
   \sum_{k=1}^{+\infty}\frac{\norm{\mathbb E\left[f\circ T^k \mid 
    \mathcal M\right]}_p}{k^{1/p}}<+\infty,
  \end{equation}
  while \eqref{MWp} holds as soon as 
  \begin{equation}\label{eq:projective_condition}
    \sum_{k=1}^{+\infty}\frac{\norm{\mathbb E\left[f\circ T^k \mid 
    \mathcal M\right]}_p}{\sqrt k}<+\infty.
   \end{equation}
 Therefore, \eqref{eq:projective_condition} gives a better 
 sufficient condition than \eqref{eq:consequence_Hannan} if we seek
 for conditions relying only on $\left(\norm{\mathbb E\left[f\circ T^k \mid 
    \mathcal M\right]}_p\right)_{k\geqslant 1}$.
    
  However, \eqref{Hannan} gives the existence of a martingale approximation 
  in the following sense: there exists a martingale difference $m
  \in\mathbb L^p(\mathcal M)$ such that 
  \begin{equation}\label{eq:martingale_approximation}
   \norm{\norm{W(n,f)-W(n,m)}_{\mathcal H_{1/2-1/p}}}_{p,\infty}=o(\sqrt n).
  \end{equation}
  Indeed, define for an integrable function $h$ and a non-negative integer $i$, $P_i(h):=
  \mathbb E\left[h\mid T^i\mathcal M\right]-
  \mathbb E\left[h\mid T^{i+1}\mathcal M\right]$. If $f$ satisfies 
  \eqref{Hannan}, then we set $m:=\sum_{i\geqslant 0}
  P_0\left(U^if\right)$. Then for any $K\geqslant 1$, the equality 
  $f-m=\sum_{i=0}^K\left(P_i(f)-P_0\left(U^if\right)\right)+\sum_{i=K+1}^{+\infty}
  \left(P_i(f)-P_0\left(U^if\right)\right)$ holds.
  Since $\sum_{i=0}^K\left(P_i(f)-P_0\left(U^if\right)\right)$ may be written as 
  $(I-U)g_K$, where $g_K$ is such that $t^p\mu\ens{\abs{g_K}>t}\to 0$ as $t$ goes to infinity, 
  we get, by inequalities (2.4) and (2.5) of \cite{MR3426520} that 
  \begin{multline*}
   \limsup_{n\to +\infty}
   \frac 1{\sqrt n}\norm{\norm{W(n,f)-W(n,m)}_{\mathcal H_{1/2-1/p}}}_{p,\infty}\\
   \leqslant \sum_{i\geqslant K+1}\limsup_{n\to +\infty}
   \frac 1{\sqrt n}\left(\norm{\norm{W(n,P_i(f)))}_{\mathcal H_{1/2-1/p}}}_{p,\infty}
   +\norm{\norm{W\left(n,P_0\left(U^i(f)\right)\right)}_{\mathcal H_{1/2-1/p}}}_{p,\infty}
   \right).
  \end{multline*}
  We conclude by Proposition~2.3 of \cite{MR3426520}. 
  
 The following condition (in the spirit of Maxwell and 
 Woodroofe's one) is sufficient for a martingale approximation in the sense of 
  \eqref{eq:martingale_approximation}:
  \begin{equation}\label{eq:MW_renforcee}
   \sum_{k=1}^{+\infty} \frac{\norm{\mathbb E[S_k(f)\mid \mathcal M}_p }{k^{1+1/p}}<+\infty.
  \end{equation}
 Indeed, Theorem~2.3 of \cite{MR3178473} gives a martingale differences sequence 
 $\left(m\circ T^i\right)_{i\geqslant 0}$ such that $
  \lim_{n\to +\infty}n^{-1/p}\norm{S_n(f-m)}_p=0$. Using Serfling 
  arguments (see \cite{MR0268938}), we get that \eqref{eq:MW_renforcee} implies 
 \begin{equation}\label{eq:martingale_approximation_Cuny_Merlevede}
   \lim_{n\to +\infty}
 n^{-1/p}\norm{\max_{1\leqslant i\leqslant n}\abs{S_i(f-m)}}_p=0.
 \end{equation}
 
 Note that for a function $h$, by Lemma~A.2 of \cite{MR3020943},
 $n^{-1/2}\norm{\norm{W(n,h)}_{\mathcal H_{1/2-1/p}}}_{p,\infty}
 \leqslant 2n^{-1/p}\norm{\max_{1\leqslant j\leqslant n}\abs{S_j(f)}}_{p,\infty}$, hence 
 by \eqref{eq:martingale_approximation_Cuny_Merlevede}, the martingale approximation 
 \eqref{eq:martingale_approximation} holds.

 Furthermore, using the construction given in \cite{MR2446326,MR2475604}, in 
 any ergodic dynamical system of positive entropy one can construct a function 
 satisfying condition \eqref{MWp} but not
 \eqref{Hannan} and vice versa. 
 \end{Remark}
 
 \begin{Remark}
  For the $\rho$-mixing coefficient defined by 
  \begin{equation*}
   \rho(n)=\sup\ens{\mathrm{Cov}(X,Y)/(\norm{X}_2\norm{Y}_2 ), 
   X \in\mathbb L^2(\sigma(f\circ T^i,i\leqslant 0),
   Y\in \mathbb L^2(\sigma(f\circ T^i, i\geqslant n)) }, 
  \end{equation*}
  Lemma~1 of \cite{MR2255301} shows that for an adapted process, 
  condition \eqref{MWp} is satisfied if the series $\sum_{n=1}^\infty 
  \rho^{2/p}(2^n)$ converges. However, the conclusion of 
  Theorem~\ref{thm:main_result} holds if $t^p\mu\ens{\abs f >t}\to 0$ 
  and $\sum_{n=1}^\infty \rho(2^n)$ converges (see Theorem~2.3, 
  \cite{1409.8169}), which is less restrictive. 
 \end{Remark}

 It turns out that even in the adapted case, condition \eqref{MWp} is sharp 
 among conditions on $\norm{\mathbb E[S_k(f)\mid \mathcal M}_p$ in the 
 following sense.
 
  \begin{Theorem}\label{optimality}
   For each sequence $(a_n)_{n\geqslant 1}$ converging to $0$ and each 
   real number $p>2$, there exists a strictly 
   stationary sequence $(f\circ T^j)_{j\geqslant 0}$ and a sub-$\sigma$-algebra 
   $\mathcal M$ such that $T\mathcal M\subset\mathcal M$, 
    \begin{equation}\label{MW_altered}
  \sum_{n=1}^{\infty}\frac{a_n}{n^{3/2} }\norm{\mathbb E[S_n(f)\mid  
   \mathcal M] }_p<\infty,   
 \end{equation}
  but the sequence $\left(n^{-1/2}W(n,f,t)\right)_{n\geqslant 1}$ is not 
  tight in $\mathcal H_{1/2-1/p}$.
  \end{Theorem}

  \begin{Remark}
   Using the inequalities  in \cite{MR2255301} in order to bound 
   $\norm{\mathbb E\left[S_n(f)\mid  
   T\mathcal M\right]}_2$, we can see that the constructed $f$ in the 
   proof of Theorem~\ref{optimality} satisfies the classical 
   Maxwell and Woodroofe condition \eqref{MW_adapted} (the fact that $p$ is strictly 
   greater than $2$ is crucial),  hence the weak invariance principle in 
   the space of continuous functions takes place.
   
   However, it remains an open question whether condition~\eqref{MW_altered} 
   implies the central limit theorem or the weak invariance 
   principle (in the space of continuous functions).
  \end{Remark}

\section{Proofs}

We may observe that condition \eqref{MWp} implies by Theorem~1 of \cite{MR2255301}
that the sequence $\left(S_n(f)/\sqrt n\right)_{n\geqslant 1}$ is bounded in $\mathbb L^p$; 
nevertheless the counter-example given in Theorem~2.6 of \cite{1409.8169} shows that 
we cannot deduce the weak invariance principle from this. 

We shall rather work with a tighness criterion.
The analogue of the continuity modulus in $C[0,1]$ is $\omega_\alpha$, defined by 
  \begin{equation*}
  \omega_\alpha(x,\delta)=\sup_{0<\abs{t-s}<\delta}\frac{\abs{x(t)-x(s)}}{\abs{t-
  s}^\alpha},\quad  x\colon [0,1]\to \R, \delta\in (0,].
  \end{equation*}
  Define $\mathcal H_\alpha^o[0,1]:=
  \ens{x\in \mathcal H_\alpha[0,1],
  \lim_{\delta\to 
  0}\omega_\alpha(x,\delta)=0}$. We shall essentially work with the space 
  $\mathcal H_\alpha^o[0,1]$ which, 
  endowed with $\norm\cdot_\alpha\colon x\mapsto  \omega_\alpha(x,1)+\abs{x(0)}$, 
  is a separable Banach space 
  (while $\mathcal H_\alpha[0,1]$ is not). Since the canonical 
  embedding $\iota\colon \mathcal H^o_\alpha[0,1]\to \mathcal H_\alpha[0,1]$
  is continuous, each 
  convergence in distribution in $\mathcal H_\alpha^o[0,1]$ also takes place in 
  $\mathcal H_\alpha[0,1]$. 

  Let us state the tighness criterion we shall use (Theorem~13 of \cite{MR1736910}).
\begin{Proposition}\label{propo:tightness}
Let $\alpha\in (0,1)$.
 A sequence of processes $(\xi_n)_{n\geqslant 1}$ with paths in $\mathcal H_\alpha^o[0,1]$ 
 and such that $\xi_n(0)=0$ for each $n$ is tight in $\mathcal H_\alpha^o[0,1]$ if and 
 only if 
 \begin{equation}\label{eq:tightness_criterion}
 \forall \varepsilon>0,\quad \lim_{\delta\to 0}\sup_{n\to +\infty}
  \mu\ens{\omega_\alpha\left(\xi_n,\delta\right)>\varepsilon}=0.
 \end{equation}
\end{Proposition}
 
In order to prove the weak convergence in $\mathcal H_\alpha^o[0,1]$, it suffices to 
prove the convergence of the finite dimensional distributions and establish tighness 
in this space. 

\subsection{A maximal inequality} 

For $p>2$, we define  
 \begin{equation}
  \norm{h}_{p,\infty}:=\sup_{\mathclap{\substack{A\in\mathcal F\\ 
  \mu(A)>0} } }\:\frac 1{\mu(A)^{1-1/p} }\mathbb E[\abs h\mathbf 1_A].     
 \end{equation}
 This norm is linked to the tail function of $h$ by the following inequalities 
 (see Exercice 1.1.12 p.~13 in \cite{MR3243734}):
 \begin{equation}\label{eq:comparison_norm_p_infty}
  \left(\sup_{t>0}t^p\mu\ens{\abs h>t }\right)^{1/p}  
  \leqslant \norm{h}_{p,\infty}\leqslant\frac p{p-1} 
  \left(\sup_{t>0}t^p\mu\ens{\abs h>t }\right)^{1/p}.
 \end{equation}
 As a consequence, if $N$ is an integer and $h_1,\dots,h_n$ are functions, then  
 \begin{equation}\label{weak_lp_norm_max}
  \norm{\max_{1\leqslant j\leqslant N}\abs{h_j}}_{p,\infty} 
 \leqslant \frac p{p-1} N^{1/p} 
 \max_{1\leqslant j\leqslant N}\norm{\abs{h_j}}_{p,\infty}.
 \end{equation}

For a positive $n\geqslant 1$, a function $f\colon\Omega\to\R$ and a 
measure-preserving map $T$, we define 
\begin{equation}
 M(n,f,T):=\max_{0\leqslant i<j\leqslant n}\frac{\abs{S_j(T,f)-S_i(T
 ,f)}}{(j-i)^{1/2-1/p}}.
\end{equation}
 By Lemma~A.2 of \cite{MR3020943}, the H\"olderian norm of a polygonal line is reached 
 at two vertices, hence  
 \begin{equation}\label{eq:link_M_normW}
  M(n,f,T)=n^{1/2-1/p}\norm{W(n,f,T,\cdot)}_{\mathcal H_{1/2-1/p}}
 \end{equation}

 Applying Proposition~2.3 of \cite{MR3426520}, we can find for each 
 $p>2$ a constant $C_p$ depending only on $p$ such that if $(m\circ T^i)_{i\geqslant 1}$ 
 is a martingale difference sequence, then for each $n$, 
 \begin{equation}\label{maximal_inequality_martingale}
  \frac 1{\sqrt n}\norm{\norm{W(n,m,T,\cdot)}_{\mathcal H_{1/2-1/p}}}_{p,\infty} 
  \leqslant C_p\norm{m}_p.
 \end{equation}
 
 In the sequel, fix such a constant $C_p$ that we shall choose greater than 
 $6\cdot 2^{1/p}p/(p-1)$. 
 We denote by $U$ the Koopman operator 
 associated with $T$, that is, 
 for each $f\colon\Omega\to\R$ and each $\omega\in \Omega$, $(Uf)(\omega)=f(T\omega)$.
  
 \begin{Definition}\label{dfn:definition_condition_C}
 Let $H$ be a closed subspace of $\mathbb L^p$. 
  Let $P$ be a linear operator from $H$ to 
  itself. We say that $\left(H,P\right)$ satisfies 
  condition $(C)$ if 
  \begin{enumerate}
   \item\label{item:inclusiosn} the inclusion 
   $U^{-1}H\subset H$ holds (respectively the inclusion 
   $UH\subset H$ holds);
   \item\label{item:power-bounded} $P$ is power bounded on $H$, that is, for each $h\in H$, 
   \begin{equation}
   K(P):= \sup_{n\geqslant 1}\sup_{h\in H\setminus\ens{0}}\frac{\norm{P^nh}_p}{\norm{h}_p} 
   <+\infty~;
   \end{equation}
   \item\label{item:PT_martingale} if $h\in H$ is such that $Ph=0$, then 
   the sequence $(h\circ T^i)_{i\geqslant 0}$ is a martingale 
   difference sequence with respect to the filtration 
   $\left(T^{-i}\mathcal M\right)_{i\geqslant 0}$ 
   (respectively
   $\left(T^{-i-1}\mathcal M\right)_{i\geqslant 0}$); 
   \item\label{item:PTU=I} $PU^{-1}f=f$ for each $f\in H$  (respectively 
   $PUf=f$ for each $f\in H$).
  \end{enumerate}
 \end{Definition}
 
 Let us give two examples of subspace $H$ and operator $P$ satisfying condition (C). 
 \begin{enumerate}
  \item Let $H$ be the subspace of $\mathbb L^p$ which consists of 
  $\mathcal M$-measurable functions and $Ph:=\mathbb E\left[Uh\mid 
  \mathcal M\right]$. Then $\left(H,P\right)$ satisfies condition (C). 
  
  \item Let $H$ be the subspace of $\mathbb L^p$ which consists of 
   functions $h$ such that $\mathbb E\left[h\mid \mathcal M\right]=0$ 
   and $Ph:=U^{-1}h-\mathbb E\left[U^{-1}h\mid 
  \mathcal M\right]$. Then $\left(H,P\right)$ satisfies condition (C). 
 \end{enumerate}

 The goal of this subsection is to establish the following maximal inequality. 
 
\begin{Proposition}\label{propo_maximal_inequality}
 Let  $T\colon\Omega\to\Omega$ be a bijective and bi-measurable measure-preserving map. 
 Let $H$ be a closed subspace of $\mathbb L^p$. Let 
 $r$ be a positive integer. For each  
, operator $P$ from $H$ to 
  itself such that $\left(H,P\right)$ satisfies condition $(C)$, each $f\in H$ 
 and each integer $n$ satisfying $2^{r-1}\leqslant n<2^r$, 
 \begin{equation}\label{maximal_inequality}
  \norm{M(n,f,T)}_{p,\infty} 
  \leqslant C_pn^{1/p}\left(\left(1+K\left(P\right)\right)\norm{f}_p 
+K_p\sum_{j=0}^{r-1}2^{-j/2} 
  \norm{\sum_{i=0}^{2^j-1}P^if}_p    \right),
 \end{equation}
 where $K_p=2^{1/p-1/2}+2^{1/2}\left(1+K(P)\right)$.
\end{Proposition} 

If $H$ is a closed subspace of $\mathbb L^p$ and  
 $P\colon H\to H$ an operator such that $\left(H,P\right)$
satisfies condition $(C)$, we define for $f\in H$ the quantity
\begin{equation}
 \norm{f}_{\mathrm{MW}(p,P)}:=
 \sum_{j=0}^{+\infty}2^{-j/2}\norm{\sum_{i=0}^{2^j-1}P^if}_p
\end{equation}
and the vector space 
\begin{equation}
 \mathrm{MW}(p,P):=\ens{f\in H\mid \norm{f}_{\mathrm{MW}(p,P)}<+\infty}.
\end{equation}
Note that $\mathrm{MW}(p,P)$ endowed with $\norm{\cdot}_{\mathrm{MW}(p,P)}$ 
is a Banach space. 

Combining Proposition~\ref{propo_maximal_inequality} and \eqref{eq:link_M_normW}, 
we derive the following bound for the Hölderian norm of the partial sum process. 

\begin{Corollary}\label{cor:bound_norm_psp}
  Let $H$ be a closed subspace of $\mathbb L^p$
  and let $P$ be an operator from $H$ to itself such that 
  $\left(H,P\right)$ satisfies the condition $(C)$. Then 
  there exists a constant $C=C(p,P)$ such that 
 for each $n$, and each $h\in H$,
 \begin{equation}
  \norm{\norm{\frac 1{\sqrt n}W(n,h)}_{\mathcal H_{1/2-1/p}}}_{p,\infty}
  \leqslant C\norm{h}_{\mathrm{MW}(p,P)}
 \end{equation}
\end{Corollary}

The proof of Proposition~\ref{propo_maximal_inequality} is in 
the same spirit as the proof of Theorem~1 of \cite{MR2255301}, which 
is done by dyadic induction. To do so, we start from the following lemma:

\begin{Lemma}\label{lem_rec_relation} 
For each positive integer $n$, each function $h\colon\Omega\to \R$ and each 
measure-preserving map $T\colon\Omega\to\Omega$, the following inequality 
holds:
\begin{equation}\label{M(n,h,T)_rec} 
 M(n,h,T)\leqslant 6\max_{0\leqslant k\leqslant n}\abs{h\circ T^k}+ 
\frac 1{2^{1/2-1/p}}M\left(\left[\frac n2\right],h+h\circ T,T^2\right).  
\end{equation}
 
\end{Lemma}
 
\begin{proof}
 First, notice that if $1\leqslant j\leqslant n$, then $j=2\left[\frac{j}2 \right]$ 
 or $j=2\left[\frac{j}2 \right]+1$, hence 
 \begin{equation}
  \abs{S_j(h)-S_{2\left[\frac{j}2 \right]}(h) }  
  \leqslant \max_{0\leqslant k \leqslant n }\abs{h\circ T^k}.
 \end{equation}
 Similarly, we have 
 \begin{equation}
  \abs{S_i(h)-S_{2\left[\frac{i+2}2 \right]}(h)} 
  \leqslant 2\max_{0\leqslant k \leqslant n }\abs{h\circ T^k}.
 \end{equation}
 It thus follows that 
 \begin{equation}\label{bound_M(n,h,T)}
  M(n,h,T)\leqslant 4\max_{0\leqslant k\leqslant n}\abs{h\circ T^k}
  +\max_{0\leqslant i<j\leqslant n}\frac{\abs{S_{2\left[\frac{j}2 \right]}(h)-
  S_{2\left[\frac{i+2}2 \right]}(h)} }{(j-i)^{1/2-1/p}}.   
 \end{equation}
Notice that if $j\geqslant i+4$, then 
\begin{equation}
 1\leqslant \left[\frac j2\right]-\left[\frac{i+2}2\right]
 \leqslant \frac{j-i}2,    
\end{equation}
 and we derive the bound 
 \begin{align*}
  \max_{0\leqslant i<j\leqslant n}\frac{\abs{S_{2\left[\frac{j}2 \right]}(h)-
  S_{2\left[\frac{i+2}2 \right]}(h)} }{(j-i)^{1/2-1/p}} &\leqslant 
  \frac 1{2^{1/2-1/p}}\max_{0\leqslant u<v\leqslant \left[\frac n2\right]}
  \frac{\abs{S_{v}(T^2,h+h\circ T)-
  S_{u }(T^2,h+h\circ T)} }{(v-u)^{1/2-1/p}} +\\
  &+\quad\max_{\mathclap{\substack{0\leqslant i<j\leqslant n\\ j\leqslant i+4} } }\quad
  \abs{S_{2\left[\frac{j}2 \right]}(h)-S_{2\left[\frac{i+2}2 \right]}(h)}.
 \end{align*}
 Since for $j\leqslant i+4$, the number of terms of the form $h
 \circ T^q$ involved in 
 $S_{2\left[\frac{j}2 \right]}(h)-S_{2\left[\frac{i+2}2 \right]}(h)$ is 
 at most $2$, we conclude that 
 \begin{align*}
  \max_{0\leqslant i<j\leqslant n}\frac{\abs{S_{2\left[\frac{j}2 \right]}(h)-
  S_{2\left[\frac{i+2}2 \right]}(h)} }{(j-i)^{1/2-1/p}} &\leqslant 
  \frac 1{2^{1/2-1/p}}M\left(\left[\frac n2\right], h+h\circ T,T^2\right) +\\
  &+2\max_{0\leqslant k\leqslant n}\abs{h\circ T^k}.
 \end{align*}
 Combining this inequality with \eqref{bound_M(n,h,T)}, we obtain 
 \eqref{M(n,h,T)_rec}, which concludes the proof of 
 Lemma~\ref{lem_rec_relation}. 
\end{proof}
 
 Now, we establish inequality \eqref{maximal_inequality} 
 by induction on $r$. 
 
 \begin{proof}[Proof of Proposition~\ref{propo_maximal_inequality}] 
  We first assume that $PU^{-1}=\mathrm{Id}$ and $U^{-1}H\subset H$.
  We check the case $r=1$. Then necessarily $n=1$ and the expression 
  $M(n,f,t)$ reduces to $f$. Since $C_p$ and $K_p$ are greater than $1$, the result 
  is a simple consequence of the triangle inequality applied to 
  $f-U^{-1}Pf$ and $U^{-1}Pf$. 
  
  Now, assume that Proposition~\ref{propo_maximal_inequality} holds 
  for some $r$ and let us show that it takes place for $r+1$. 
  We thus consider an integer $n$ such that $2^r\leqslant n<2^{r+1}$, 
  a function $f\in H$, a measure-preserving map 
 $T\colon\Omega\to\Omega$ bijective and bi-measurable, and a sub-$\sigma$-algebra 
 $\mathcal M$ satisfying $T\mathcal M \subset \mathcal M$, a closed subspace $H$ of 
 $\mathbb L^2$ such that 
 $U^{-1}H\subset H$ and an operator $P\colon 
 H\to H$ such that $\left(H,P\right)$
 satisfies condition $(C)$ with $PU^{-1}=\mathrm{Id}$ and we have to 
 show that \eqref{maximal_inequality} holds with $r+1$ instead of $r$. 
  First, using inequality $M(n,f)\leqslant 
  M(n,f-U^{-1}Pf)+M(n,U^{-1}Pf)$ and 
  Lemma~\ref{lem_rec_relation} with $h :=U^{-1}Pf$, we derive 
  \begin{multline}
   M(n,f,T)\leqslant M\left(n,f-U^{-1}Pf,T\right)+6\max_{0\leqslant k\leqslant n}
  \abs{U^{-1}Pf\circ T^k}+\\
  +\frac 1{2^{1/2-1/p} }M\left(\left[\frac n2 
  \right],(I+U)U^{-1}Pf,T^2\right),
  \end{multline}
   hence taking the norm $\norm{\cdot}_{p,\infty}$, we obtain 
   by \eqref{weak_lp_norm_max} that 
  \begin{multline}\label{induction_step}
   \norm{M(n,f,T)}_{p,\infty}  \leqslant \norm{M(n,f-U^{-1}Pf,T)}_{p,\infty}+
   6(n+1)^{1/p}\frac p{p-1} \norm{U^{-1}Pf}_p+ \\ 
  +\frac 1{2^{1/2-1/p} }\norm{M\left(\left[\frac n2 
  \right],(I+U)U^{-1}Pf,T^2\right)}_{p,\infty}.
  \end{multline} 
   By inequality \eqref{maximal_inequality_martingale} and accounting the fact that 
   $6\cdot(n+1)^{1/p}p/(p-1)\leqslant C_pn^{1/p} $, we obtain 
   \begin{multline}\label{induction_step2}
   \norm{M(n,f,T)}_{p,\infty}  \leqslant C_pn^{1/p} \norm{f-U^{-1}Pf}_p+
   C_pn^{1/p}\norm{Pf}_p  +\\
  +\frac 1{2^{1/2-1/p} }\norm{M\left(\left[\frac n2 
  \right],(I+U)U^{-1}Pf,T^2\right)}_{p,\infty}.
   \end{multline}
 Since $2^{r-1}\leqslant \left[n/2\right]< 2^r$, we may apply the 
 induction hypothesis to the integer $\left[n/2\right]$, the function 
 $h:=(I+U^{-1})Pf$, $T^2$ instead of $T$ and $P^2$ instead 
 of $P$. This 
 gives 
  \begin{multline}
  \left[\frac n2 
  \right]^{-1/p} \norm{M\left(\left[\frac n2 
  \right],h,T^2\right)}_{p,\infty}
  \leqslant C_p\left(1+K\left(P^2\right)\right)\norm{h} _p+\\ 
  +C_p\widetilde{K_p}\sum_{j=0}^{r-1}2^{-j/2}\norm{ 
  \sum_{i=0}^{2^j-1}P^{2i}\left(I+U^{-1}\right)Pf}_p,    
  \end{multline}
  where $\widetilde{K_p}=2^{1/p-1/2}+2^{1/2}\left(1+K(P^2)\right)$.
 Notice that $\norm h_p 
  \leqslant 2\norm{Pf}_p$, and by item \ref{item:PTU=I} of 
  Definition~\ref{dfn:definition_condition_C}, 
  it follows that 
  \begin{equation}
   \sum_{i=0}^{2^j-1}P^{2i}\left(I+U^{-1}\right)Pf=
    \sum_{i=0}^{2^j-1}\left(P^{2i+1}f+P^{2i}f\right)=
    \sum_{i=0}^{2^{j+1}-1}P^if.
  \end{equation}

  Accounting 
  the inequality $K\left(P^2\right)\leqslant K\left(P\right)$ and 
  $\widetilde{K_p}\leqslant K_p$, we have 
  \begin{align*}
  \left[\frac n2 
  \right]^{-1/p} \norm{M\left(\left[\frac n2 
  \right],h,T^2\right)}_{p,\infty}
  &\leqslant 2\left(1+K\left(P\right)\right)C_p\norm{Pf}_p+C_pK_p\sum_{j=0}^{r-1}2^{-j/2}\norm{ 
 \sum_{i=0}^{2^{j+1}-1}P^if}_p\\ 
 &=2\left(1+K\left(P\right)\right)C_p\norm{Pf}_p+ 2^{1/2} C_pK_p\sum_{j=1}^{r}2^{-j/2}\norm{ 
 \sum_{i=0}^{2^{j}-1}P^if}_p 
  \end{align*}
  and we infer 
  \begin{multline}
 \norm{M\left(\left[\frac n2 
  \right],h,T^2\right)}_{p,\infty}\leqslant  \left(\frac n2\right)^{1/p} 
  \left(2\left(1+K\left(P\right)\right)-K_p\sqrt 2)\right)C_p\norm{Pf}_p\\ 
  +n^{1/p}  2^{1/2-1/p} C_pK_p\sum_{j=0}^{r}2^{-j/2}\norm{ 
 \sum_{i=0}^{2^{j}-1}P^if}_p.
  \end{multline}
 Pluggling this into \eqref{induction_step2}, we derive 
 \begin{multline}
   \norm{M(n,f,T)}_{p,\infty}  \leqslant C_pn^{1/p}\left(1+K\left(P\right)\right) 
   \norm{f}_p +
  n^{1/p} C_pK_p\sum_{j=0}^{r}2^{-j/2}\norm{\sum_{i=0}^{2^{j}-1}P^if}_p+
   \\ +C_pn^{1/p}\left(1+2^{1-1/p}(1+K(P))-2^{1/2-1/p}K_p\right)\norm{Pf}_p.
 \end{multline}
 The definition of $K_p$ implies that $2^{1/p-1/2}-\sqrt 2(1+K(P))-K_p=0$, hence 
 \eqref{maximal_inequality} is established. This concludes the 
 proof of Proposition~\ref{propo_maximal_inequality} in the case 
 $PU^{-1}=\mathrm{Id}$. 
  
 When $PU=\mathrm{Id}$ and $UH\subset H$ we do the same proof, but replacing each occurrence of 
 $U^{-1}$ by $U$. This ends the proof of Proposition~\ref{propo_maximal_inequality}. 
 \end{proof}
 
\subsection{Proof of Theorem~\ref{thm:main_result}}\label{mart_approx} 

Since the 
convergence of the finite dimensional distributions is contained in 
 the main result of \cite{MR2344817}, the only difficulty in 
 proving Theorem~\ref{thm:main_result} is to establish tightness. 
To this aim, we shall proceed as in the proof of 
Theorem~5.3 in \cite{1403.0772}.  

\begin{Proposition}\label{propo:tightness_under_condition_C}
 Let $T$ be a measure preserving map, $H$ a closed subspace of $\mathbb L^p$ ($p>2$) 
 and let $P$ be an  
 operator from $H$ to itself such that $\left(H,P\right)$ satisfies 
 condition $(C)$. 
 Assume that $h$ is an element of $H$ such that $\norm{h}_{\mathrm{MW}\left(p,P\right)}
 <+\infty$ 
 
 Then the sequence $(n^{-1/2}W(n,h))_{n\geqslant 1}$ is tight in 
 $\mathcal H_{1/2-1/p}$. 
\end{Proposition}
\begin{proof}
 Let us define $V_n:=\sum_{i=0}^{n-1}P^i$. Using 
 $\norm{V_nV_k}_p\leqslant K(P)\min\ens{k\norm{V_n}_p,
 n\norm{V_k}_p}$, we derive that for each $f\in\mathrm{MW}(p,P)$,
 \begin{equation}
  \frac{\norm{V_{2^n}f}_{\mathrm{MW}(p,P)}}{2^n}\leqslant 
  K(P)\left(\frac{\norm{V_{2^nf}}_p}{2^{n/2}}+
  \sum_{k\geqslant n+1}\frac{\norm{V^{2^k}f}_p}{2^{k/2}}\right)
 \end{equation}
 which goes to $0$ as $n$ goes to infinity. If $m\geqslant 1$ is an 
 integer and if $n$ is such that $2^n\leqslant m<2^{n+1}$, then 
 \begin{equation}
  \frac{\norm{V_mf}_{\mathrm{MW}(p,P)}}m\leqslant 
  \frac{K(P)}m\sum_{k=0}^n\norm{V_{2^k}f}_{\mathrm{MW}(p,P)}
  \leqslant \frac{K(P)}m\sum_{k=0}^n2^k\varepsilon_k,
 \end{equation}
 where $(\varepsilon_k)_{k\geqslant 1}$ is a sequence converging to $0$. 
 This entails that the operator $P$ is mean-ergodic on $\mathrm{MW}(p,P)$. 
 Furthermore, since $P$ has no non trivial fixed points on the 
 Banach space $\left(\mathrm{MW}(p,P),\norm{\cdot}_{\mathrm{MW}(p,P)}\right)$, we derive by 
 Theorem~1.3 p.73 of \cite{MR797411} that the subspace 
 $(I-P)\mathrm{MW}(p,P)$ is dense in $\mathrm{MW}(p,P)$ for the 
 topology induced by the norm $\norm{\cdot}_{\mathrm{MW}(p,P)}$.
 
 Let $h\in H$ be such that $\norm{h}_{\mathrm{MW}(p,P)}<+\infty$ and 
 $x>0$. We can find $f\in (I-P)\mathrm{MW}(p,P)$ such that 
 $\norm{h-f}_{\mathrm{MW}(p,P)}<x$. Consequently, using 
 Corollary~\ref{cor:bound_norm_psp}, we derive that for each positive $\varepsilon$ 
 and $\delta$, 
 \begin{equation}\label{eq:proof_tightness_under_condition_C_step_1}
  \mu\ens{\omega_{1/2-1/p}\left(\frac 1{\sqrt n}W(n,h),\delta\right)>2\varepsilon}
  \leqslant \varepsilon^{-p}x+
  \mu\ens{\omega_{1/2-1/p}\left(\frac 1{\sqrt n}W(n,f),\delta\right)>\varepsilon}.
  \end{equation}
 Now, since the function $f$ belongs to $(I-P)\mathrm{MW}(p,P)$, we can find 
 $f'\in \mathrm{MW}(p,P)$ such that $f=f'-Pf'$. If $PU^{-1}=\mathrm{Id}$, then 
 we write $f=f'-U^{-1}Pf'+(U^{-1}-I)f'$ and if $PU=\mathrm{Id}$, then 
 $f=f'-UPf'+(U-I)f'$. In other words, $f$ admits a 
 martingale-coboundary decomposition in $\mathbb L^p$ (since $f'$ belongs to $\mathbb L^p)$.
 Consequently, by Corollary~2.5 of \cite{MR3426520}, the 
 sequence $(n^{-1/2}W(n,f))_{n\geqslant 1}$ is tight in $\mathcal H_{1/2-1/p}$. By 
 Proposition~\ref{propo:tightness} and \eqref{eq:proof_tightness_under_condition_C_step_1}, 
 we derive that for each positive $\varepsilon$ and $x$, 
 \begin{equation}
 \lim_{\delta\to 0}
 \limsup_{n\to +\infty}\mu\ens{\omega_{1/2-1/p}\left(\frac 1{\sqrt n}W(n,h),\delta\right)>2\varepsilon}
  \leqslant \varepsilon^{-p}x.
 \end{equation}
 Since $x$ is arbitrary we conclude the proof of \eqref{propo:tightness_under_condition_C}
 by using again Proposition~\ref{propo:tightness}.
\end{proof}

\begin{proof}[Proof of Theorem~\ref{thm:main_result}]
 Writing $f=\mathbb E\left[f\mid \mathcal M\right]+f-\mathbb E\left[f\mid \mathcal M\right]$, 
 the proof reduces (as mentioned in the begining of the section) to establish 
 tightness in $\mathcal H_{1/2-1/p}^o[0,1]$ of the sequences 
 $\left(W_n\right)_{n\geqslant 1}:= 
 \left(n^{-1/2}W\left(n,\mathbb E\left[f\mid \mathcal M\right]\right)\right)_{n\geqslant 1}$ and 
 $\left(W'_n\right)_{n\geqslant 1}:=
 \left(n^{-1/2}W\left(n,f-\mathbb E\left[f\mid \mathcal M\right]\right)\right)_{n\geqslant 1}$.
 
 \begin{itemize}
  \item Tightness of $\left(W_n\right)_{n\geqslant 1}$. 
  We define 
  \begin{equation}
   P(f):=\mathbb E\left[Uf\mid\mathcal M\right]
  \mbox{ and }H:=\ens{f\in\mathbb L^p,f\mbox{ is }\mathcal M\mbox{-measurable}}.
  \end{equation}

   Then 
  $\left(H,P\right)$ 
  satisfies condition $(C)$. Since 
  \begin{equation}
   \sum_{i=0}^{n-1}P^i\left(\mathbb E\left[f\mid \mathcal M\right]\right)=
  \mathbb E\left[S_n(f)\mid \mathcal M\right],
  \end{equation}

   the convergence of 
  the first series in \eqref{MWp} is equivalent to $f\in 
  \mathrm{MW}(p,P)$ (by Lemma~2.7 of \cite{MR2123210}). By 
  Proposition~\ref{propo:tightness_under_condition_C},
  we derive that the sequence $\left(W_n\right)_{n\geqslant 1}
  $ is tight in $\mathcal H_{1/2-1/p}^o[0,1]$.
  
  \item Tightness of $\left(W'_n\right)_{n\geqslant 1}$. 
  We define 
  \begin{equation}
   P(f):=U^{-1}f-\mathbb E\left[U^{-1}f\mid\mathcal M\right]
  \mbox{ and }H:=\ens{f\in\mathbb L^p,\mathbb E\left[f\mid \mathcal M\right]=0}.
  \end{equation}

   Since for each $f\in H$ and each $k\geqslant 1$, 
  $\norm{P^kf}_p\leqslant 2\norm{f}_p$, $\left(H,P\right)$
  satisfies condition $(C)$ (see the proof of Proposition~2 in 
  \cite{MR2344817} for the other conditions). Since $P\left(\mathbb E\left[f\mid \mathcal M\right]
  \right)=0$, we have 
  \begin{equation}
   \sum_{i=1}^{n}P^i\left(f-\mathbb E\left[f\mid \mathcal M\right]\right)=
   \sum_{i=1}^{n}P^if=
  U^{-n}\left(S_n(f)-\mathbb E\left[S_n(f)\mid T^{-n}\mathcal M\right]\right),
  \end{equation}
  hence 
 the convergence of 
  the second series in \eqref{MWp} implies that $f$ belongs to  
  $\mathrm{MW}(p,P)$ (by Lemma~37 of \cite{MR3077530}). By 
  Proposition~\ref{propo:tightness_under_condition_C},
  we derive that the sequence $\left(W'_n\right)_{n\geqslant 1}$ 
  is tight in $\mathcal H_{1/2-1/p}^o[0,1]$.
 \end{itemize}
This ends the proof of Theorem~\ref{thm:main_result}.
\end{proof}

\subsection{Proof of Theorem~\ref{optimality}} We take a similar construction as in the 
proof of Proposition~1 of \cite{MR2255301}. We consider a non-negative 
sequence $(a_n)_{n\geqslant 1}$, and a sequence $(u_k)_{k\geqslant 1}$ of 
real numbers such that 
\begin{equation}\label{def_u_j}
 u_1=1,u_2=2, u_k^{p/2+1}+1<u_{k+1}\mbox{ for }k\geqslant 3\mbox{ and }
 a_t\leqslant k^{-2} \mbox{ for }t\geqslant u_k.  
\end{equation}
Notice that since $p>2$, the conditions \eqref{def_u_j} are more 
restrictive than that of the 
proof of Proposition~1 of \cite{MR2255301}.
 If $i=u_j$ for some $j\geqslant 1$, then we define 
 $p_i:=cj/u_j^{1+p/2}$ and $p_i=0$ otherwise. Let $(Y_k)_{k\geqslant 0}$ be a discrete time 
 Markov chain with the state space $\Z^+$ and transition matrix 
 given by $p_{k,k-1}=1$ for $k\geqslant 1$ and $p_{0,j-1}:=p_j$, $j\geqslant 1$. 
 We shall also consider a random variable $\tau$ which takes its values among 
 non-negative integers, and whose distribution is given by $\mathbb \mu(\tau=j)
 =p_j$. Then the stationary distribution 
 exists and is given by 
 \begin{equation}
  \pi_j=\pi_0\sum_{i=j+1}^\infty p_i, j\geqslant 1, \mbox{ where }
 \pi_0=1/\mathbb E[\tau].
 \end{equation}
 We start from the stationary distribution $(\pi_j)_{j\geqslant 0}$ and we 
 take $g(x):=\mathbf 1_{x=0}-\pi_0$, where $\pi_0 =\mu\ens{Y_0=0}$. 
 We then define $f\circ T^j=X_j:=g(Y_j)$. 
 
 It is already checked in \cite{MR2255301}
 that the sequence $(X_j)_{j\geqslant 0}$ satisfies \eqref{MW_altered}, 
 where $\mathcal M=\sigma(X_k,k\leqslant j)$ and 
 $S_n=\sum_{j=1}^nX_j$. To conclude the proof, it remains to check that the 
 sequence $\left(n^{-1/2}  W(n,f,T)\right)_{n\geqslant 1} $ is not tight in 
 $\mathcal H_{1/2-1/p}^o$, which will be done by 
 disproving \eqref{eq:tightness_criterion} for a particular choice of 
 $\varepsilon$. To this aim, we define 
 \begin{equation}\label{def_t_and_tau} 
  T_0=0, T_k=\min\ens{t>T_{k-1}\mid Y_t=0 },\quad \tau_k=T_k-T_{k-1}, k\geqslant 1.   
 \end{equation}
 Then $(\tau_k)_{k\geqslant 1}$ is an independent sequence and each $\tau_k$ is
 distributed as $\tau$ and
 \begin{equation}\label{sum_of_Tk} 
  S_{T_k}=\sum_{j=1}^k(1-\pi_0\tau_j)=k-\pi_0T_k.  
 \end{equation}
 
 Let us fix some integer $K$ greater than $\mathbb E[\tau]$. 
 Let $\delta\in (0,1)$ be fixed and $n$ an integer such that $1/n<\delta$. Then 
 the inequality 
 \begin{multline}
  \frac 1{(nK)^{1/p} }\quad \max_{\mathclap{\substack{0\leqslant i<j\leqslant nK\\ 
  j-i\leqslant n\delta} }  }
  \quad\frac{\abs{S_j-S_i}}{(j-i)^{1/2-1/p} }  \geqslant 
   \frac 1{(nK)^{1/p} }\mathbf 1\ens{T_n\leqslant Kn}\times\\ 
   \times\max_{1\leqslant k 
  \leqslant n}\frac{\abs{S_{T_k}-S_{T_{k-1} }  } }{(T_k-T_{k-1})^{1/2-1/p} }
  \mathbf 1\ens{\abs{T_k-T_{k-1} }\leqslant n\delta }  
 \end{multline}
 takes place. By \eqref{def_t_and_tau} and \eqref{sum_of_Tk}, this can 
 be rewritten as 
 \begin{multline}
  \frac 1{(nK)^{1/p} }\quad \max_{\mathclap{\substack{0\leqslant i<j\leqslant nK\\ 
  j-i\leqslant n\delta} }  }
  \quad\frac{\abs{S_j-S_i}}{(j-i)^{1/2-1/p} }  \geqslant 
   \frac 1{(nK)^{1/p} }\mathbf 1\ens{T_n\leqslant Kn}\times\\ 
   \times\max_{1\leqslant k 
  \leqslant n}\frac{\abs{1-\pi_0 \tau_k  } }{\tau_k^{1/2-1/p} }
  \mathbf 1\ens{\tau_k\leqslant n\delta } .
 \end{multline}
 Defining for a fixed $C$ the event 
 \begin{equation}
  A_n(C):=\ens{\frac{\abs{1-\pi_0 \tau  } }{\tau^{1/2-1/p} }
  \geqslant  C(Kn)^{1/p}  }  \cap \ens { \tau \leqslant n \delta},
 \end{equation}
 we obtain by independence of $\left(\tau_k\right)_{k\geqslant 1}$ 
 \begin{equation}
  \mu\ens{ \frac 1{(nK)^{1/p} }\quad \max_{\mathclap{\substack{0\leqslant i<j\leqslant nK\\ 
  j-i\leqslant n\delta} }  }
  \quad\frac{\abs{S_j-S_i}}{(j-i)^{1/2-1/p} }\geqslant C}\geqslant 
   1- \left(1- \mu(A_n(C)) \right)^n 
  - \mu\ens{T_n> Kn}.
 \end{equation}
 By the law of large numbers, we obtain, accounting $K > \mathbb E[\tau]$, that 
 \begin{equation}\label{eq:weakened_MW_counter_example_step}
  \limsup_{n \to \infty }\mu\ens{ \frac 1{(nK)^{1/p} }\quad 
   \max_{\mathclap{\substack{0\leqslant i<j\leqslant nK\\ 
  j-i\leqslant n\delta} }  }
  \quad\frac{\abs{S_j-S_i}}{(j-i)^{1/2-1/p} }\geqslant C}
   \geqslant\limsup_{n \to \infty }1- \left(1- \mu(A_n(C)) \right)^n.
 \end{equation}
We choose $C:=\pi_0/(2K^{1/p})$. 
Considering the integers $n$ of the form $\left[u_j^{(p+2)/2}\right]$, we obtain in view of 
\eqref{eq:weakened_MW_counter_example_step}~:
\begin{multline}\label{eq:weakened_MW_counter_example_step2}
 \limsup_{n\to \infty} \mu\ens{ \frac 1{(nK)^{1/p} }\quad 
 \max_{\mathclap{\substack{0\leqslant i<j\leqslant nK\\ 
  j-i\leqslant n\delta} }  }
  \quad\frac{\abs{S_j-S_i}}{(j-i)^{1/2-1/p} }\geqslant \frac{\pi_0}{2K^{1/p}}}\geqslant \\
 \geqslant \limsup_{j\to \infty} 
 1- \left(1- \mu\left(A_{\left[u_j^{(p+2)/2}\right]}\left(\frac{\pi_0}{2K^{1/p}}\right)
 \right) \right)^{\left[u_j^{(p+2)/2}\right]}.
 \end{multline}
 Since $\tau\geqslant 1$ almost surely, 
the following inclusions take place for $n>(2/\pi_0)^p$:
\begin{align*}
 A_n(\pi_0/(2K^{1/p}))&\supset \ens{\pi_0\tau^{1/2+1/p}-\tau^{-1/2+1/p}
 \geqslant \pi_0/(2K^{1/p})(Kn)^{1/p}}\cap \ens { \tau \leqslant n \delta}\\
 &\supset \ens{\tau^{1/2+1/p}
 \geqslant \frac{1+\pi_0n^{1/p}/2}{\pi_0}}\cap \ens { \tau \leqslant n \delta}\\
 &\supset \ens{\tau^{1/2+1/p}
 \geqslant n^{1/p}}\cap \ens { \tau \leqslant n \delta}\\
 &=\ens{n^{2/(p+2)}\leqslant \tau \leqslant n \delta}.
\end{align*}
Consequently, for $j$ large enough,  
\begin{equation}
 \mu\left(A_{\left[u_j^{(p+2)/2}\right]}\left(\frac{\pi_0}{2K^{1/p}}\right)\right)
 \geqslant \mu\ens{\left[u_j^{(p+2)/2}\right]^{2/(p+2)}\leqslant \tau \leqslant 
 \left[u_j^{(p+2)/2}\right] \delta}.
\end{equation}
Since $\tau$ takes only integer values among $u_l$'s and $\left[u_j^{(p+2)/2}\right] \delta
<u_{j+1}$ (by \eqref{def_u_j} and the fact that $\delta<1$), we obtain in view 
of \eqref{eq:weakened_MW_counter_example_step2}, that 
\begin{multline}\label{eq:weakened_MW_counter_example_step3}
 \limsup_{n\to \infty} \mu\ens{ \frac 1{(nK)^{1/p} }\quad 
 \max_{\mathclap{\substack{0\leqslant i<j\leqslant nK\\ 
  j-i\leqslant n\delta} }  }
  \quad\frac{\abs{S_j-S_i}}{(j-i)^{1/2-1/p} }\geqslant \frac{\pi_0}{2K^{1/p}}}\geqslant \\
 \geqslant \limsup_{j\to \infty}1- \left(1-\mu\ens{\tau=u_j}\right)^{\left[u_j^{(p+2)/2}\right] }\\
 =1-\liminf_{j\to \infty}\left(1-cju_j^{-1-p/2}\right)^{\left[u_j^{(p+2)/2}\right] }.
\end{multline}
Noticing that for a fixed $J$, 
\begin{equation}
 \liminf_{j\to \infty}\left(1-cju_j^{-1-p/2}\right)^{\left[u_j^{(p+2)/2}\right] }
 \leqslant \limsup_{j\to \infty}
 \left(1-cJu_j^{-1-p/2}\right)^{\left[u_j^{(p+2)/2}\right] }=e^{-cJ},
\end{equation}
we deduce that the 
last term of \eqref{eq:weakened_MW_counter_example_step3} is equal to $1$. Since 
\begin{equation}
  \frac 1{(nK)^{1/p} }\quad 
 \max_{\mathclap{\substack{0\leqslant i<j\leqslant nK\\ 
  j-i\leqslant n\delta} }  }
  \quad\frac{\abs{S_j-S_i}}{(j-i)^{1/2-1/p}}\leqslant \omega_{1/2-1/p}
  \left(\frac 1{\sqrt{nK}}W(nK,f),\delta\right),
\end{equation}
we derive that \eqref{eq:tightness_criterion} does not hold with $\varepsilon=
\pi_0/(2K^{1/p})$.
This finishes the proof of Theorem~\ref{optimality}.

\bigskip 

\textbf{Acknowledgements} The author would like to thank an anonymous referee 
for many valuable comments which improved the presentation of the paper and 
led to a shorter proof of Theorem~\ref{thm:main_result}.

\def\polhk\#1{\setbox0=\hbox{\#1}{{\o}oalign{\hidewidth
  \lower1.5ex\hbox{`}\hidewidth\crcr\unhbox0}}}\def\cprime{$'$}
\providecommand{\bysame}{\leavevmode\hbox to3em{\hrulefill}\thinspace}
\providecommand{\MR}{\relax\ifhmode\unskip\space\fi MR }
\providecommand{\MRhref}[2]{%
  \href{http://www.ams.org/mathscinet-getitem?mr=#1}{#2}
}
\providecommand{\href}[2]{#2}

\end{document}